\documentclass{amsart}
\usepackage{amsmath,amsthm}
\usepackage{amsfonts,amssymb}
\usepackage{enumerate}
\usepackage{graphicx}

\newcommand{\wt}{\widetilde{\omega}}

 \def\a{\alpha}
 \def\b{\beta}
 
  \def\i{\infty}
 
 \def\o{\omega}
 \def\t{\tilde}
 \def\th{\theta}
 \def\d{\delta}
 
 \def\l{\lambda}
 \def\ep{\epsilon}
 \def\m{\hspace{-1mm}}
 \def\be{\begin{eqnarray*}}
 \def\ee{\end{eqnarray*}}
 
 \newcommand{\N}{{\mathbb N}}
\newtheorem{lemma}{Lemma}[section]

\newtheorem{corollary}[lemma]{Corollary}
\newtheorem{theorem}[lemma]{Theorem}

\def\be  {\begin{equation}}
\def\ee  {\end{equation}}
\def\ba  {\begin{eqnarray}}
\def\ea  {\end{eqnarray}}
\def\baa {\begin{eqnarray*}}
\def\eaa {\end{eqnarray*}}
\def\bc  {}
\makeatletter
\@addtoreset{equation}{section}
\makeatother
\def\qed{\hfill $\Box$\smallskip}
\begin{document}

\title[The largest critical value of $T_n^{(k)}$]{Estimates for the largest critical value of $T_n^{(k)}$}

\author{Nikola Naidenov}
\address{Faculty of Mathematics and Informatics, Sofia University "St. Kliment Ohridski", 5 James Bourchier Blvd.,
1164 Sofia, Bulgaria}
\email{nikola@fmi.uni-sofia.bg}

\author{Geno Nikolov}
\address{Faculty of Mathematics and Informatics, Sofia University "St. Kliment Ohridski", 5 James Bourchier Blvd.,
1164 Sofia, Bulgaria}
\email{geno@fmi.uni-sofia.bg}

\thanks{The authors are partially supported by the Sofia University Research Fund through Contract No. 80-10-20/22.03.2021. }

\begin{abstract}
Here we study the quantity
$$
\tau_{n,k}:=\frac{|T_n^{(k)}(\omega_{n,k})|}{T_n^{(k)}(1)}\,,
$$
where $T_n$ is the $n$-th Chebyshev polynomial of the first kind and
$\omega_{n,k}$ is the largest zero of $T_n^{(k+1)}$. Since the 
absolute values of the local extrema of $T_n^{(k)}$
increase monotonically towards the end-points of $[-1,1]$, the value
$\tau_{n,k}$ shows how small is the largest critical value of
$\,T_n^{(k)}\,$ relative to its global maximum $\,T_n^{(k)}(1)$. 
This is a continuation of the recent paper \cite{NNS2018}, where upper 
bounds and asymptotic formuae for $\tau_{n,k}$ have been obtained 
on the basis of Alexei Shadrin's explicit form of the Schaeffer--Duffin pointwise 
majorant for polynomials with absolute value not exceeding $1$ in $[-1,1]$. 

We exploit a result of Knut Petras \cite{KP1996} about the weights of the Gaussian 
quadrature formulae associated with the ultraspherical weight function
$w_{\lambda}(x)=(1-x^2)^{\lambda-1/2}$ to find an explicit  
(modulo $\omega_{n,k}$) formula for $\tau_{n,k}^2$.  This enables us to prove  a
lower bound and to refine the upper bounds for $\tau_{n,k}$ obtained in 
\cite{NNS2018}. The explicit formula admits also a new derivation 
of the assymptotic formula in \cite{NNS2018} approximating $\tau_{n,k}$ for $n\to\infty$. 
The new approach is simpler, without using deep results about the ordinates of 
the Bessel function, and allows to better analyze the sharpness of the estimates. 
\medskip

\medskip

\end{abstract}

\maketitle
\section{Introduction and statement of the results}
We study the quantity
$$
\tau_{n,k}:=\frac{|T_n^{(k)}(\omega)|}{T_n^{(k)}(1)}\,,
$$
where $T_n$ is the $n$-th Chebyshev polynomial of the first kind, and
$\omega=\omega_{n,k}$ is the rightmost zero of $T_n^{(k+1)}$, $n\geq k+2$. The value
$\tau_{n,k}$ shows how small is the largest critical value of
$\,T_n^{(k)}\,$ relative to its global maximum in $[-1,1]$, $\,T_n^{(k)}(1)$.
This quantity has found applications in studying some extremal problems such as 
Markov-type inequalities \cite{ES1942}, \cite{Nik2005}, \cite{AS2004} and 
the Landau--Kolmogorov inequalities for intermediate derivatives \cite{BOE1998}, \cite{AS2014}.

Some upper bounds for $\tau_{n,k}$ have been obtained in the recent paper \cite{NNS2018}. The main ingredient for the results in \cite{NNS2018} is  the pointwise majorant $D_{n,k}(x)$ for polynomials of degree at most $n$ with absolute value less than or equal to one in $[-1,1]$. This majorant was used by Schaeffer and Duffin \cite{SD1938} to obtain another  proof of V. Markov's inequality. An explicit formula for $D_{n,k}^2(x)$, $\,k\geq 2$, was found by Shadrin (see \cite{AS2004}), it reads 
\begin{equation}\label{e1.1}
D_{n,k}^2(x)
=\frac{n^2(n^2-1^2)\cdots(n^2-(k-1)^2)}{(1-x^2)^{k}}\,S_{n,k}(x)\,,
\end{equation}
where 
$$
S_{n,k}(x) =1+\sum_{m=1}^{k-1}\frac{(2m-1)!!}{(2m)!!}
\,\frac{(k-m)_{2m}}{(1-x^2)^{m}}
\prod_{j=1}^{m}\frac{1}{n^2-j^2} 
$$
and $(a)_j:=a(a+1)\ldots(a+j-1)$, $j\in \mathbb{N}$, is the Pocchammmer function. 
This leads to the inequality
\begin{equation}\label{e1.2}
\tau_{n,k}^2\leq\frac{D_{n,k}^2(\omega)}{\big[T_n^{(k)}(1)\big]^2}=
\frac{(2k-1)!!}{(2k)!!}\,\frac{n+k}{n}\,\frac{1}{{n+k\choose n-k}}\,\frac{S_{n,k}(\omega)}
{(1-\omega^2)^k }\,,
\end{equation}
which was the starting point in \cite{NNS2018} for the derivation of upper bounds for 
$\tau_{n,k}^2$ which are uniform in $n$ and $k$.

Of course, one can obtain an explicit formula for $\tau_{n,k}$ from $T_n^{(k)}=2^{k-1}k!\,n\,P_{n-k}^{(k)}$ and the representations of ultraspherical polynomials as hypergeometric functions, e.g., \cite[eqn. 4.7.6)]{GS1975} yields
$$
\tau_{n,k}=-1+\sum_{m=1}^{n-k}(-1)^{m+1}{n-k\choose m}\frac{(n+k)_{m}}{(k+1/2)_m}\Big(\frac{1-\omega}{2}\Big)^{m}\,.
$$
However, this expression is difficult to estimate because of the sign changing summands. 
It turns out that an explicit formula for $\tau_{n,k}^2$ exists, which moreover admits easier  estimation. We prove the following
\begin{theorem}\label{t1.1}
For all $n>k+1$, the quantity $\,\tau_{n,k}^2\,$ admits the representation
\begin{equation}\label{e1.3}
\tau_{n,k}^2=\frac{(2k-1)!!}{(2k)!!}\,\frac{n}{n-k}\,\frac{1}{{n+k\choose n-k}}\,\frac{1}{(1-\omega^2)^k\,S_{n,k+1}(\omega)}\,,
\end{equation}
where
$$
S_{n,k+1}(x) =1+\sum_{m=1}^{k}\frac{(2m-1)!!}{(2m)!!}
\,\frac{(k+1-m)_{2m}}{(1-x^2)^{m}}
\prod_{j=1}^{m}\frac{1}{n^2-j^2}\,.
$$
\end{theorem}

We derive Theorem~\ref{t1.1} from a result of Knut Petras \cite{KP1996}, who has found explicit expressions for the coefficients of the Gaussian quadrature formulae associated with the ultraspherical weight function $w_{\lambda}(x)=(1-x^2)^{\lambda-1/2}$ when $\lambda$ in a non-negative integer.

The similarity of formulae \eqref{e1.2} and \eqref{e1.3} is remarkable, as well as the opposite roles
played by the quantities  $S_{n,k}(\omega)$ and $S_{n,k+1}(\omega)$ therein. As \eqref{e1.3} is the exact expression while in \eqref{e1.2} we have a majorant for $\tau_{n,k}^2$, it is natural to expect that with  \eqref{e1.3} one could produce  better upper bounds than the bounds obtained with \eqref{e1.2}. This however requires a lower estimate for $S_{n,k+1}(\omega)$, exactly as upper estimates for $S_{n,k}(\omega)$ are needed in \eqref{e1.2}. In fact, one can avoid estimation of $S_{n,k}(\omega)$ and 
$S_{n,k+1}(\omega)$  by simply combining \eqref{e1.2} and \eqref{e1.3} to obtain 
\begin{equation}\label{e1.4}
\tau_{n,k}^2\leq
\sqrt{\frac{n+k}{n-k}}\,\frac{(2k-1)!!}{(2k)!!}\frac{1}{{n+k\choose
n-k}}\,\frac{1}{(1-\omega^2)^{k}}\,\sqrt{\frac{S_{n,k}(\omega)}{S_{n,k+1}(\omega)}} 
\end{equation}
and then use the inequality $S_{n,k}(\omega)/S_{n,k+1}(\omega)<1$, which is easily verified by a termwise comparison. By elaborating further this idea, we prove the following sharper result.
\begin{theorem}\label{t1.2}
For all $n>k+1$, there holds
\begin{equation}\label{e1.5}
\tau_{n,k}\leq \frac{(2k-1)!!}{(n+k-1)(n+k-3)\ldots(n-k+1)}\,\frac{1}{(1-\omega^2)^{\frac{k}{2}}}\,.
\end{equation}
\end{theorem}
Since $\displaystyle{1-\o^2\geq \Big(\frac{k+2}{n}}\Big)^2\,$ (see, e.g.,  \cite[Lemma 3.5]{Nik2005} for an estimate for the largest zeros of ultraspherical polynomials), Theorem~\ref{t1.2} implies:
\begin{corollary}\label{c1.1}
For all $n>k+1$, there holds
\begin{equation*}
\tau_{n,k}\leq \frac{(2k-1)!!}{(n+k-1)(n+k-3)\ldots(n-k+1)}\,\Big(\frac{n}{k+2}\Big)^{k}\,.
\end{equation*}
\end{corollary}

We proved in \cite[Theorem 1.1]{NNS2018})  that the sequence $\{\tau_{n,k}\}_{n>k+1}$ is monotonically decreasing, thus showing the existence of 
$\,\tau_k^{*}=\lim_{n\to\infty}\tau_{n,k}$\,. Corollary~\ref{c1.1} implies immediately an upper bound for $\tau_{k}^{*}$.
\begin{corollary}\label{c1.2}
For all $k\in\mathbb{N}$, the quantity $\tau_k^{*}$ satisfies the inequality
\begin{equation*}
\tau_{k}^{*}\leq \frac{(2k-1)!!}{(k+2)^{k}}\,.
\end{equation*}
\end{corollary}

The following counterpart of the inequality $\,\displaystyle{\frac{\tau_{n,k}}{\tau_{n+1,k}}>1}\,$  holds true:
\begin{theorem}\label{t1.2'}
For all $n>k+1$, 
\begin{equation}\label{e1.6'}
\frac{\tau_{n,k}}{\tau_{n+1,k}}\leq \frac{(n+k)(n+k-2)\ldots(n-k+2)}{(n+k-1)(n+k-3)\ldots(n-k+1)}\,.
\end{equation}
\end{theorem}

By iterating \eqref{e1.6'} and using that  $\tau_{k+2,k}=1/(2k+1)$ (see \cite[Theorem 1.2]{NNS2018}),  we obtain that  the  right-hand side of \eqref{e1.5} without the factor $(1-\omega^2)^{-k/2}$ is a lower bound for $\tau_{n,k}$. 

\begin{corollary}\label{c1.3}
For all $n>k+1$, there holds
\begin{equation}\label{e1.6}
\tau_{n,k}\geq \frac{(2k-1)!!}{(n+k-1)(n+k-3)\ldots(n-k+1)}\,.
\end{equation}
\end{corollary}

We observe that the bounds for $\tau_{n,k}$  given by Corollaries \ref{c1.1} and \ref{c1.3} present correctly the magnitude of $\tau_{n,k}$ whenever $\Big(\dfrac{n}{k+2}\Big)^{k}$ remains bounded. This is the case, e.g., when $m=n-k$ is fixed. 

Theorem \ref{t1.1} admits also a new derivation of the asymptotic formula for $\tau_k^{*}$, obtained in  \cite[Theorem 1.8]{NNS2018}. The new approach is simpler, without using deep results about the ordinates of the Bessel function, and allows to better analyze the sharpness of our estimates. 

\begin{theorem}\label{t1.4}
The quantity $\tau_k^{*}$ admits the representation
\begin{equation}\label{e1.7}
\tau_k^*= A\,\Big({2\over e}\Big)^{k+1/2} e^{-a k^{1/3}}k^{-1/6} \Big(1 + O\big(k^{-1/6}\big) \Big)\, ,
\end{equation}
where  $\displaystyle{A= \left(\int_0^{\infty} e^{-\frac{x^3}{3}-2ax}\,\frac{dx}{\sqrt{\pi x}} \right)^{-1/2}}\approx 1.3951~$ and $~a=2^{-1/3}|i_1|\approx 1.8558\,$ with $i_1$ -- the first zero of the Airy function.
\end{theorem}

The rest of the paper is organized as follows. In Section 2 we prove Theorem~\ref{t1.1}. Theorems \ref{t1.2} and \ref{t1.2'} are proven in Section 3, and the proof of Theorem~\ref{t1.4} is given in Section~4. Section 5 contains comments and concluding remarks.
\section{Proof of Theorem \ref{t1.1}}
The derivatives of $T_n$ are expressed by ultraspherical polynomials, namely,
\begin{equation}\label{e2.1}
T_n^{(k)}=n\,2^{k-1}(k-1)!\,P_{n-k}^{(k)}\,, \quad k=1,\ldots,n\,.
\end{equation}
Here, $P_{m}^{(\lambda)}$ is the usual notation for the $m$-th ultraspherical polynomial, which is orthogonal in $[-1,1]$ with respect to the weight function $w_{\lambda}(x)=(1-x^2)^{\lambda-1/2}$, $\lambda>-1/2$. Well known properties of ultraspherical polynomials are
\begin{equation}\label{e2.2}
(1-x^2)y^{\prime\prime}-(2\lambda+1)xy^{\prime}+m(m+2\lambda)y=0,\quad y=P_m^{(\lambda)}\,,
\end{equation}
\begin{equation}\label{e2.3}
\frac{d}{dx}\big\{P_m^{(\lambda)}(x)\big\}=2\lambda\,P_{m-1}^{(\lambda+1)}(x)\,,
\end{equation}
\begin{equation}\label{e2.4}
P_m^{(\lambda)}(1)={m+2\lambda-1\choose m}\,.
\end{equation}
From \eqref{e2.1} and \eqref{e2.4} it follows that
\begin{equation}\label{e2.5}
\tau_{n,k}^2=\frac{\big[P_{n-k}^{(k)}(\o)\big]^2}{\big[P_{n-k}^{(k)}(1)\big]^2}=\frac{\Gamma^2(2k)\,\Gamma^2(n+1-k)}{\Gamma^2(n+k)}\,\big[P_{n-k}^{(k)}(\o)\big]^2\,,
\end{equation}
where $\o$ is the largest zero of $P_{n-k-1}^{(k+1)}$. 

An explicit representation of $\big[P_{n-k}^{(k)}(\o)\big]^2$ follows from a result of Knut Petras \cite{KP1996}, where the author has found an asymptotic expansion for the coefficients of the Gaussian quadrature formulae associated with the ultraspherical weight function $w_{\lambda}(x)=(1-x^2)^{\lambda-1/2}$. In the case $\lambda\in \mathbb{N}_0$ Petras has proved that
the coefficients $a_{\nu,n}^{(\lambda)}$ of the $n$-point Gaussian quadrature formula $Q_{n}^{(\lambda)G}$,  
$$
Q_{n}^{(\l)G}[f]=\sum_{\nu=1}^{n}a_{\nu,n}^{(\l)}\,f(x_{\nu,n}^{(\l)})\,,
$$
where $x_{\nu,n}=x_{\nu,n}^{(\l)}$ are the zeros of $P_n^{(\lambda)}$, admit the representation
\begin{equation}\label{e2.6}
a_{\nu,n}^{(\l)}=\frac{\pi}{n+\lambda}(1-x_{\nu,n}^2)^{\l}\,\Bigg(1+\sum_{m=1}^{\lambda-1}
\frac{\alpha_{m}(\lambda)}{(1-x_{\nu,n}^2)^m}\,\prod_{j=1}^{m}\frac{1}{(n+\l)^2-j^2}\Bigg),
\end{equation}
where
$$
\alpha_{m}(\lambda)=\Bigg(\frac{(2m)!}{2^mm!}\Bigg)^2{m+\l-1\choose 2m}\,.
$$
On the other hand, the weights $a_{\nu,n}^{(\l)}$ obey the representation (cf. \cite[eqn. (15.3.2)]{GS1975})
$$
a_{\nu,n}^{(\l)}=\frac{2^{2-2\l}\pi\,\Gamma(n+2\l)}{\Gamma^2(\l)\,\Gamma(n+1)}\,\frac{1}{(1-x_\nu^2)\big[P_n^{(\l)\prime}(x_{\nu})\big]^2}\,,
$$
If $\l>1$, then, by \eqref{e2.3}, $\,P_n^{(\l)\prime}(x)=y^{\prime\prime}(x)/(2\l-2)$, where $y=P_{n+1}^{(\l-1)}$. By \eqref{e2.2},  $(1-x^2)y^{\prime\prime}(x)=-(n+1)(n+2\l-1)y(x)$  at the zeros of $y^{\prime}=2(\l-1)P_n^{(\l)}$, therefore the above formula can be rewritten as
\begin{equation}\label{e2.7}
a_{\nu,n}^{(\l)}=\frac{2^{4-2\l}\pi\,\Gamma(n+2\l-1)}{(n+1)(n+2\l-1)\,\Gamma^2(\l-1)\,\Gamma(n+2)}\,\frac{1-x_{\nu}^2}{\big[P_{n+1}^{(\l-1)}(x_{\nu})\big]^2}\,,\quad \l>1.
\end{equation}
By equating the right-hand sides of \eqref{e2.6} and \eqref{e2.7} and then substituting $\l=k+1$ and replacing $n$ by $n-k-1$, where $1\leq k<n-1$, we obtain 
$$
\frac{\pi}{n}(1-x_{\nu}^2)^{k+1}\,S_{n,k+1}(x_{\nu})=
\frac{2^{2-2k}\pi\,\Gamma(n+k)}{(n^2-k^2)\Gamma^2(k)\Gamma(n-k+1)}\,\frac{1-x_{\nu}^2}
{\big[P_{n-k}^{(k)}(x_{\nu})\big]^2}\,.
$$
In particular, this last equality holds true when $x_{\nu}$ is the largest zero of $P_{n-k-1}^{(k+1)}$, i.e.,  $x_{\nu}=\o$. Therefore,
$$
\big[P_{n-k}^{(k)}(\o)\big]^2 =
\frac{2^{2-2k}n\,\Gamma(n+k)}{(n^2-k^2)\Gamma^2(k)\Gamma(n-k+1)}\,\frac{1}
{(1-\o^2)^{k}\,S_{n,k+1}(\o)}\,.
$$
By putting this expression in \eqref{e2.5} we obtain \eqref{e1.3}.\qed
\section{Proof of Theorems \ref{t1.2} and \ref{t1.2'}}
Recall that the ${}_3F_2$ hypergeometric function is defined by the series 
$$
_3F_2(a,b,c;d,e;x)=1+\sum_{m=1}^{\infty}\frac{(a)_m(b)_m(c)_m}{(d)_m(e)_m}\,\frac{x^m}{m!} \,.
$$
Generally, it is assumed that $\,d,\,e$ are not negative integers or zero, but exceptions are allowed when some of parameters $a, b, c$ is a negative integer, in which case the series terminates. 
This is the situation with the finite sums  $\,S_{m,k}(0)\,$ and $\,S_{m,k+1}(0)\,$, where $m\in\mathbb{N}$, $m> k+1$, namely, we have
\begin{eqnarray*}
&&S_{m,k}(0)={ _3}F_2\big(k,1-k,\frac{1}{2};1+m,1-m;1\big)\,,\\
&&S_{m,k+1}(0)={  _3}F_2\big(k+1,-k,\frac{1}{2};1+m,1-m;1\big)\,.
\end{eqnarray*}
A closed type formula for such ${ _3}F_2$ expressions provides the Whipple identity (see \cite[p. 189, eqn. (7)]{AE1953})
$$
_3F_2\big(a,1-a,c;d,2c+1-d;1\big)=\frac{2^{1-2c}\pi\Gamma(d)\Gamma(2c+1-d)} {\Gamma\big(\frac{a+d}{2}\big)\Gamma\big(\frac{a+1+2c-d}{2}\big)}\,\frac{1}{\Gamma\big(\frac{1-a+d}{2}\big)\Gamma\big(\frac{2+2c-a-d}{2}\big)}\,.
$$
By using Whipple's identity and familiar properties of the Gamma function (considering separately the cases of even and odd $m-k$), we find that, under the assumption $m>k+1$, 
\begin{eqnarray}
&&S_{m,k}(0)=m\,\frac{(m+k-2)(m+k-4)\cdots(m-k+2)}{(m+k-1)(m+k-3)\cdots(m-k+1)}\,,
\label{e3.1}\\
&&S_{m,k+1}(0)=m\,\frac{(m+k-1)(m+k-3)\cdots(m-k+1)}{(m+k)(m+k-2)\cdots(m-k)}\,.
\label{e3.2}
\end{eqnarray}

For the proof of Theorems~\ref{t1.2} and \ref{t1.2'} we need the following simple lemma.
\begin{lemma}\label{l3.1}
Let the polynomials $\,P(x)=\sum_{m=0}^{k}a_m x^m\,$ and $\,Q(x)=\sum_{m=0}^{k}b_m x^m\,$ have positive coefficients (with $a_k$ allowed to  be zero). If the sequence $\,\{\frac{a_m}{b_m}\}_{m=0}^{k}\,$ is monotonically increasing (resp. decreasing), then $\,R(x)=\frac{P(x)}{Q(x)}\,$ is strictly monotonically increasing (resp. decreasing) in $[0,\infty)$\,.
\end{lemma}

\begin{proof}
A straightforward calculation shows that $\,R^{\prime}(x)=r(x)/Q^2(x)$, where 
\begin{equation*}
\begin{split}
r(x)=&\sum_{s=1}^{k}\Bigg(\sum_{m=0}^{\lfloor\frac{s-1}{2}\rfloor}
(s-2m)(a_{s-m}b_m-a_m b_{s-m}\Bigg)x^{s-1}\\
&+\sum_{s=k+1}^{2k}\Bigg(\sum_{m=s-k}^{\lfloor\frac{s-1}{2}\rfloor}
(s-2m)(a_{s-m}b_m-a_m b_{s-m}\Bigg)x^{s-1}\,.
\end{split}
\end{equation*}
(Here, $\lfloor\cdot\rfloor$ stands for the integer part function.) 
We observe that if the sequence $\,a_m/b_m$, $(m=0,1,\ldots,k),\,$ is monotonically increasing (decreasing), then all the coefficients of the polynomial  $\,r(x)\,$ are positive (negative), and hence $\,R^{\prime}(x)\,$ is positive (negative) on $[0,\infty)$.
\end{proof}

\noindent
{\em Proof of Theorem~\ref{t1.2}.} In the introduction we deduced inequality \eqref{e1.4} by combining \eqref{e1.2} and Theorem~\ref{t1.1}. Now we refine the trivial estimate $S_{n,k}(\o)/S_{n,k+1}(\o)<1$. Let us set 
\begin{equation}\label{e3.3}
z=\frac{1}{1-x^2}\,,\qquad z\in [1,\infty)\,
\end{equation}
Consider the polynomials $\,P(z)=S_{n,k}(x)\,$ and $\,Q(z)=S_{n,k+1}(x)\,$, where  $\,x\in [0,1)$. The coefficients $\,\{a_m\}\,$ and  $\,\{b_m\}\,$ of $P(z)$ and $Q(z)$, respectively, are
\begin{eqnarray*}
&&a_m=\frac{(2m-1)!!}{(2m)!!}\,(k-m)_{2m}\,\prod_{j=1}^m\frac{1}{n^2-j^2}\,,\\
&&b_m=\frac{(2m-1)!!}{(2m)!!}\,(k+1-m)_{2m}\,\prod_{j=1}^m\frac{1}{n^2-j^2}\,,
\end{eqnarray*}
The sequence 
$$
\frac{a_m}{b_m}=\frac{k-m}{k+m},\qquad m=0,1,\ldots,k, 
$$
is monotonically decreasing. By Lemma~\ref{l3.1}, $\,P(z)/Q(z)\,$ is monotonically decreasing in the interval $[0,\infty)$ and therefore in $[1,\infty)$. It follows from \eqref{e3.3} and the definition of $\,P\,$ and $\,Q\,$ that $\,S_{n,k}(x)/S_{n,k+1}(x)\,$ is a monotonically decreasing function of $\,x\,$ in the interval $\,[0,1)$, hence
$$
\frac{S_{n,k}(\omega)}{S_{n,k+1}(\omega)}\leq \frac{S_{n,k}(0)}{S_{n,k+1}(0)}\,.
$$
By using \eqref{e3.1} and \eqref{e3.2} we find
$$
\sqrt{\frac{S_{n,k}(0)}{S_{n,k+1}(0)}}=\sqrt{n^2-k^2}\,\frac{(n+k-2)(n+k-4)\cdots(n-k+2)}{(n+k-1)(n+k-3)\cdots(n-k+1)}\,.
$$
and putting the last expression in the right-hand side of \eqref{e1.4},  after some simplification we arrive at inequality \eqref{e1.5}.\qed 

{\em Proof of Theorem~\ref{t1.2'}.} From Theorem~\ref{t1.1} we have
\begin{equation*}
\frac{\tau_{n,k}^2}{\tau_{n+1,k}^2}=\frac{n(n+1+k)}{(n+1)(n-k)}\, 
\frac{(1-\wt^2)^kS_{n+1,k+1}(\wt)}{(1-\omega^2)^kS_{n,k+1}(\omega)}\,,
\end{equation*}
where $\,\wt\,$ is the largest zero of $\,T_{n+1}^{(k+1)}$. Since  
$\,(1-\wt^2)^kS_{n+1,k+1}(\wt)\,$ is a polynomial in $\,1-\wt^2\,$ with positive coefficents and $0<1-\wt^2<1-\o^2$,  it follows that 
\begin{equation}\label{e3.4}
\frac{\tau_{n,k}^2}{\tau_{n+1,k}^2}\leq \frac{n(n+1+k)}{(n+1)(n-k)}\, 
\frac{(1-\omega^2)^kS_{n+1,k+1}(\omega)}{(1-\omega^2)^kS_{n,k+1}(\omega)}\,.
\end{equation}
Let us consider the polynomials in $z=1-\o^2$, $z\in (0,1]$,  
\begin{eqnarray*}
&&P(z)=z^k+\sum_{m=0}^{k-1}a_m z^m=(1-\o^2)^kS_{n+1,k+1}(\o)\,,\\ 
&&Q(z)=z^k+\sum_{m=0}^{k-1}b_m z^m=(1-\o^2)^kS_{n,k+1}(\o)\,.
\end{eqnarray*}
For $\,m=1,\ldots, k$ we have 
\begin{eqnarray*}
&&a_{k-m}=\frac{(2m-1)!!}{(2m)!!}\,(k-m+1)_{2m}\,\prod_{j=1}^{m}\frac{1}{(n+1)^2-j^2}\,,
\\
&&b_{k-m}=\frac{(2m-1)!!}{(2m)!!}\,(k-m+1)_{2m}\,\prod_{j=1}^{m}\frac{1}{n^2-j^2}\,, 
\end{eqnarray*}
therefore
$$
\frac{a_{k-m}}{b_{k-m}}=\frac{n^2-m^2}{(n+1)^2-m^2}\,\frac{a_{k+1-m}}{b_{k+1-m}}<\frac{a_{k+1-m}}{b_{k+1-m}}\,,\quad m=1,\ldots, k. 
$$
Hence, the sequence $\{a_m/b_m\}$, $m=0,1,\ldots,k$, is monotonically increasing, and  Lemma~\ref{l3.1} implies that $\frac{P(x)}{Q(x)}$ increases monotonically in $(0,\infty)$, in particular,  
\begin{equation}\label{e3.5}
\frac{(1-\omega^2)^kS_{n+1,k+1}(\omega)}{(1-\omega^2)^kS_{n,k+1}(\omega)}=
\frac{P(z)}{Q(z)}\leq\frac{P(1)}{Q(1)}= \frac{S_{n+1,k+1}(0)}{S_{n,k+1}(0)}\,.
\end{equation}
From \eqref{e3.2} we find
\begin{equation}\label{e3.6}
 \frac{S_{n+1,k+1}(0)}{S_{n,k+1}(0)}=\frac{(n+1)(n-k)}{n(n+k+1)}\,\frac{(n+k)^2(n+k-2)^2\ldots(n-k+2)^2}{(n+k-1)^2(n+k-3)^2\ldots(n-k+1)^2}\,.
\end{equation}
The claim of Theorem~\ref{t1.2'} now follows from \eqref{e3.4}, \eqref{e3.5} and \eqref{e3.6}. \qed

\section{Proof of Theorem \ref{t1.4}} 
In view of \eqref{e1.3} we have $(\tau_k^*)^2= L_1/L_2$, where 
$$ 
L_1= \lim_{n\to\infty} \frac{n(n-k-1)!(2k-1)!!^2}{(n+k)!(1-\omega^2)^k}\quad \text{ ~and~}\quad L_2= \lim_{n\to\infty}S_{n,k+1}(\omega)\,. 
$$

We will use the following result from \cite{GS1975}, (see \S 8.9 or Theorem~8.21.12).

Let $\a > -1$ and $\b$ be an arbitrary real number. Then, for the $r$-th zero of $P^{(\a,\b)}_n(\cos\th)$, where $P^{(\a,\b)}_n(x)$ is the Jacobi polynomial, it holds the limit relation
$$
\th_r= n^{-1}\big(j_{\a,r} + \ep_n), \quad \ep_n\to 0 {\rm ~~ for ~~} n\to\i ,$$
where $j_{\a,r}$ is the $r$-th positive zero of the Bessel function $J_\a(x)$.

Since $\o$ is the largest zero of $T_n^{(k+1)}(x)= C_{n,k} P_{n-k-1}^{(\nu,\nu)}(x)$ with $\nu= k+1/2$, then 
$$ 
\o= \cos\th_1= \cos{j_{\nu,1}+\ep_{n-k-1} \over n-k-1} = 1 - {j_{\nu,1}^2 \over 2n^2}(1 + \ep'_n), 
$$
where $\ep'_n= \ep'(n,k)$ tends to $0$ as $n\to\i$ and $k$ is fixed. Equivalently, we have
\begin{equation}\label{e5.1}
1-\o^2 = \Big({j_{\nu,1}\over n}\Big)^2(1 + \d_n), \qquad \d_n\to 0 \quad \text{ as } 
\quad n\to\i .  
\end{equation}
For $L_1$ we obtain
\begin{equation*}
\begin{split}
L_1&= \lim_{n\to\i} {(2k-1)!!^2\over (n^2-1^2)(n^2-2^2) \cdots (n^2-k^2) (1-\o^2)^k}\\
&= \lim_{n\to\i} {(2k-1)!!^2\over (1-1^2/n^2)(1-2^2/n^2) \cdots (1-k^2/n^2) j_{\nu,1}^{2k}(1+\d_n)^k}\,,
\end{split}
\end{equation*}
hence
\begin{equation}\label{e5.2}
L_1 = {(2k-1)!!^2\over j_{\nu,1}^{2k}} \,.
\end{equation}
For $L_2=\lim_{n\to\i} S_{n,\kappa}(\o)$,  $\kappa=k+1$,  we have
\begin{equation*}
\begin{split}
L_2& = \lim_{n\to\i} \sum_{m=0}^k {(2m-1)!!\over (2m)!!}\,{\kappa\over \kappa+m}\,{(\kappa^2-1^2)(\kappa^2-2^2) \cdots (\kappa^2-m^2) \over (n^2-1^2)(n^2-2^2) \cdots (n^2-m^2)} (1-\o^2)^{-m}\\
&= \lim_{n\to\i} \sum_{m=0}^k {(2m\!-\!1)!!\over (2m)!!}\,{(\kappa)^{2m+1}\over \kappa+m}\,
\Big(\prod_{\ell=1}^{m}\frac{1-\ell^2/\kappa^2}{1-\ell^2/n^2}\Big)\, j_{\nu,1}^{-2m}
(1+\d_n)^{-m} 
\end{split}
\end{equation*}
therefore
\begin{equation*}
L_2= \sum_{m=0}^k {(2m-1)!!\over (2m)!!}\,{\kappa\over \kappa+m}\, \Big(\prod_{\ell=1}^m \big(1-{\ell^2/\kappa^2}\big)\Big)\,\Big({j_{\nu,1}\over \kappa}\Big)^{-2m}\,,
\end{equation*}
hence
\begin{eqnarray}\label{e5.3}
&&L_2= \sum_{m=0}^k a_m\,q^{2m}=:S_{k+1}^{*} \label{e5.3}\\
&&a_m= a_{m,k}= {(2m-1)!!\over (2m)!!}\,{(k+m)!\over(k-m)!}(k+1)^{-2m},\quad q=\Big({j_{\nu,1}\over k+1}\Big)^{-1}\nonumber
\end{eqnarray}

Combining \eqref{e5.2} and \eqref{e5.3} we obtain 
\begin{equation}\label{e5.4}
\tau_k^*= {(2k-1)!!\over j_{\nu,1}^{k}\,\sqrt{S^*_{k+1}}}\,,
\qquad \nu= k+{1\over2} \, . 
\end{equation}

Notice that  $\displaystyle{J_{k+{1\over2}}(z)= \sqrt{2\over\pi} z^{k+{1\over2}}\Big(-{1\over z}{d\over dz} \Big)^k {\sin z\over z}}$, therefore $\;j_{\nu,1}$ is a zero of an elementary function.

To estimate the factor $\displaystyle{{(2k-1)!!\over j_{\nu,1}^{k}\,}}$ in \eqref{e5.4} we use the Stirling  approximation $\displaystyle{N!=  \sqrt{2\pi N}\Big({N\over e}\Big)^N\Big(1 + O\big(N^{-1}\big)\Big)}$ and (see, e.g., \cite{QW1999})
$$ j_{\nu,1}= \nu + a\nu^{1/3} + {3a^2\over 10}\nu^{-1/3} + O\big(\nu^{-1}\big) , \quad \nu>0. $$
Then,
\begin{equation*}
\begin{split}
{(2k-1)!!\over j_{\nu,1}^{k}} &= {(2k)!\over (2k)!!}\nu^{-k}\Big[1 + a\nu^{-2/3} +  O\big(\nu^{-4/3}\big)\Big]^{-k} \\
&= 2^{-k}{(2k)!\over k!}\Big(k + {1\over2}\Big)^{-k}\exp\left\{-k\log\Big[1 + a k^{-2/3} +  O\big(k^{-4/3}\big)\Big]\right\} \\
&= 2^{-k}\sqrt2{(2k/e)^{2k}\over (k/e)^k}\Big(1\! +\! O\big(k^{-1}\big)\Big) k^{-k}\Big(1\! +\! {1\over2k}\Big)^{-k}\!\exp\left\{-a k^{1/3}\! +\!  O\big(k^{-1/3}\big)\right\} \\
&= \sqrt2(2/e)^{k}e^{-1/2}\Big(1+O\big(k^{-1}\big)\Big) e^{-a k^{1/3}}\Big(1 + O\big(k^{-1/3}\big)\Big)\,,
\end{split} 
\end{equation*}
hence
\begin{equation}\label{e5.5}
{(2k-1)!!\over j_{\nu,1}^{k}}= \Big({2\over e}\Big)^{k+1/2} e^{-a k^{1/3}}\Big(1 + O\big(k^{-1/3}\big)\Big)\,. 
\end{equation}

The approximation of $S^*_{k+1}$ in the denominator needs more care. We start with the coefficients $a_m$ in \eqref{e5.3}. We shall use Stirling's formula in the form
$$
\log(N-1)!= \Big(N-{1\over2}\Big)\log N - N + {1\over2}\log2\pi + O\big(N^{-1}\big).
$$
With $\kappa=k+1$, we have
\begin{equation*}
\begin{split}
{(2m)!!\,a_m\over(2m\!-\!1)!!}\m =&\kappa^{-2m}\! \exp\Bigg\{\!\Big[\Big(\kappa\! +\! m\! -\! {1\over2}\Big)\log(\kappa\!+\!m)\! -\! (\kappa\!+\!m)\! +\! {1\over2}\log2\pi\! +\!  O\Big({1\over \kappa\!+\!m} \!\Big)\!\Big]\\
&\qquad\qquad\; - \Big[\Big(\kappa\! -\! m\! -\! {1\over2}\Big)\log(\kappa\!-\!m)\! -\! (\kappa\!-\!m) \!+\! {1\over2}\log2\pi\! +\!  O\Big({1\over \kappa\!-\!m}\! \Big)\!\Big]\!\Bigg\} 
\end{split}
\end{equation*}
\begin{equation*}
\begin{split}
= &\kappa^{-2m}\! \exp\Big\{\!\log \kappa\Big[\!\Big(\!\kappa\! +\! m\! -\! {1\over2}\!\Big)\! -\! \Big(\!\kappa\! -\! m\! -\! {1\over2}\!\Big)\!\Big]\!  +\! \Big(\!\kappa\! +\! m\! -\! {1\over2}\Big)\log\Big(\!1\! +\! {m\over \kappa}\Big)\\
&\qquad\qquad\ \;-\Big(\!\kappa\!-\! m\! -\! {1\over2}\!\Big)\log\Big(1\! -\! {m\over \kappa}\!\Big)
\! -\! 2m \!+\!  O\Big({1\over \kappa\!-\!m}\Big)\!\Big\} \\
=& \exp\Big\{\!\Big(\!\kappa\! +\! m\! - \!{1\over2}\!\Big)\!\sum_{j=1}^\i\! {(\!-\!1)^{j\!-\!1}\over j} \Big(\!{m\over \kappa}\!\Big)^j \!+\! \Big(\kappa\!-\! m\! -\! {1\over2}\!\Big)\!\sum_{j=1}^\i \!{1\over j} \Big({m\over \kappa}\Big)^j\! - \!2m\! +\!  O\Big(\!{1\over \kappa\!-\!m}\!\Big)\!\Big\} \\
=& \exp\Bigg\{\!\sum_{odd ~j>0} {(2\kappa\!-\!1)\over j} \Big({m\over \kappa}\Big)^j\! -\! \sum_{even ~j>0} {(2m)\over j} \Big({m\over \kappa}\Big)^j\! -\! 2m\! +\!  O\Big({1\over \kappa\!-\!m}\Big)\!\Bigg\} \\
=& \exp\Bigg\{\!\sum_{ even \;i\geq\!0} {2m\over \!i+\!1} \Big(\!{m\over \kappa}\!\Big)^i 
\!-\!{1\over2}\log{1\! +\! {m\over\kappa}\over1 \!- \!{m\over\kappa}}\! -\!\! \sum_{even \;j\!>\!0}\!\! {2m\over j} \Big({\!m\over \kappa}\!\Big)^j \!-\! 2m\! +\!  O\Big({1\over \kappa\!-\!m}\Big)\!\Bigg\} \\
=& \sqrt{\kappa-m\over \kappa+m} \exp\left\{-2m\sum_{even ~j>0} \Big({1\over j} - {1\over j+1}\Big) \Big({m\over \kappa}\Big)^j +  O\Big({1\over \kappa-m}\Big)\right\} \\
=& \sqrt{\kappa-m\over \kappa+m} \exp\left\{-{m^3\over 3\kappa^2} - O\Big({m^5\over \kappa^4}\Big) + O\Big({1\over \kappa-m}\Big)\right\}\,.
\end{split}
\end{equation*}

It is important that the remainder denoted by $\displaystyle{- O\Big({m^5\over \kappa^4}\Big)}$ is negative. The same holds true for the other "O" term, but we will not use this fact.

Next, for the ratio $q$ in \eqref{e5.3} we have
\begin{equation*}
\begin{split}
q &= \Big({j_{\nu,1}\over \kappa}\Big)^{-1}= \Big({\nu\! +\! a\, \nu^{1/3}\! + \!  O\big(\nu^{-1/3}\big) \over \nu + 1/2}\Big)^{-1}= \Big(1\! +\! {1\over2\nu}\Big) \Big(1\! +\! a\, \nu^{-2/3}\! +\! O\big(\nu^{-4/3}\big)\Big)^{-1} \\
&= \Big(1 + {1\over2\nu}\Big) \Big(1 - a\, \nu^{-2/3} + O\big(\nu^{-4/3}\big)\Big) = 1 - a\, \nu^{-2/3} + {1\over2\nu} + O\big(\nu^{-4/3}\big) \\
 &= 1 - a\,\kappa^{-2/3} + O\big(\kappa^{-1}\big) , \quad \nu\geq 1/2 \,. 
 \end{split}
\end{equation*}

From this it is clear that for sufficiently large $k$ (respectively $\kappa$ and $\nu$) we have $q\in (0,1)$. Moreover, the same holds for $\nu\geq 1/2$, as can be seen from the results in \cite{QW1999}.

Now, we split the sum $\displaystyle{S_{k+1}^*=\sum_{m=0}^k a_m\, q^{2m}}$ into three parts
$$ 
S_{k+1}^*=\sum_{m=0}^{m_1}(\cdot) + \sum_{m=m'_1}^{m_2}(\cdot) + \sum_{m=m'_2}^{k}(\cdot) =: S_{\kappa}^{(1)} + S_{\kappa}^{(2)} + S_{\kappa}^{(3)}\,,
$$
where $m_1= \lfloor \kappa^{1/3}\rfloor$, $m_2= \lfloor A_k \kappa^{2/3}\rfloor$ with $A_k= \log \kappa$ and $m'_i= m_i + 1$.

Note that without loss of generality we may assume that $k$ is sufficiently large so that the three sums above are non-empty. For small $k$ the assertion of the theorem is fulfilled on account of the choice of the constant in "O".\medskip

We estimate $S_{\kappa}^{(1)}$ from above  by  using  $\displaystyle{ {(2m-1)!!\over(2m)!!}= {1\over \sqrt{\pi m}}\Big(1+O\big(m^{-1}\big)\Big)}$.
\begin{eqnarray*} 
S_{\kappa}^{(1)} &<& \sum_{m=0}^{m_1} a_{m,k} < 1 + \sum_{m=1}^{m_1} {(2m-1)!!\over (2m)!!} \exp\left\{ O\Big({1\over \kappa-m}\Big)\right\} \\
   &=& 1 + \sum_{m=1}^{m_1} O\Big({1\over\sqrt{m}}\Big)\left(1 + O\Big({1\over \kappa}\Big)\right)= O\big( \sqrt{m_1}\big) = O\big(k^{1/6}\big) ~. 
\end{eqnarray*}

The third sum is also relatively small. Indeed,
\begin{equation*}
\begin{split} 
S_{\kappa}^{(3)} &< \sum_{m=m'_2}^{k} a_{m,k} < \sum_{m=m'_2}^{k} {(2m-1)!!\over (2m)!!}\exp\left\{ -{m^3\over 3\kappa^2} + O\big( 1 \big)\right\}\\
& < \sum_{m=m'_2}^{k} {C\over\sqrt{m}} \exp\left\{ -{m^3\over 3\kappa^2}\right\}\,, 
\end{split}
\end{equation*}
where $C$ is an absolute constant independent of $m$ and $k$. Hence,
$$ S_{\kappa}^{(3)} < k\,{C\over\sqrt{m_2}}\, e^{-m_2'^3/(3\kappa^2)} <  k\,{C\over\sqrt{A_k}\kappa^{1/3}}\, e^{-A_k^3/3} < {C\,k\over \kappa^{1/3}}\, e^{-2A_k/3} = C\,k/\kappa ~, $$
provided $A_k^2>2$, i.e. for $k\geq 4$. As a consequence, $~S_{\kappa}^{(3)}= O(1)$ for $k\in\N$.
\smallskip

For the main part of $S^*_{\kappa}$ we have
\begin{equation*}
\begin{split}
&S_{\kappa}^{(2)}\! =\! \sum_{m=m'_1}^{m_2} a_m q^{2m}
\!=\! \sum_{m=m'_1}^{m_2}\! {1\! +\! O\big( m^{-1}\big) \over \sqrt{\pi m}}\! \Big(1\! +\! O\Big({m\over \kappa}\Big)\Big)\\
&\qquad\qquad\qquad\qquad\quad\times \exp\left\{\! -{m^3\over 3\kappa^2}\! -\! O\Big({m^5\over \kappa^4}\Big) \!+\! O\Big({1\over \kappa}\Big)\right\} \Big(1\! -\! a (\kappa)^{-2/3}\! +\! O\big(\kappa^{-1}\big) \Big)^{2m} \\
&=\sum_{m=m'_1}^{m_2}\! {1 \over \sqrt{\pi m}} \Big(1\! +\! O\Big({A_k\over \kappa^{1/3}}\Big)\Big) \exp\left\{\! -{m^3\over 3\kappa^2}\! -\! O\Big({A_k^5\over \kappa^{2/3}}\Big)\! +\! 2m\Big[\!- a (\kappa)^{-2/3}\! +\! O\big(\kappa^{-1}\big) \Big] \right\}\\
&= \sum_{m=m'_1}^{m_2} {1 \over \sqrt{\pi m}} \Big(1 + O\Big({A_k\over \kappa^{1/3}}\Big)\Big) \exp\left\{ -{1\over3}\Big({m\over \kappa^{2/3}}\Big)^3 -  2a \Big({m\over \kappa^{2/3}}\Big) \right\}\\
&= \Big(1 + O\Big({A_k\over \kappa^{1/3}}\Big)\Big){1\over\sqrt{h\pi}} \sum_{m=m'_1}^{m_2} {h \over \sqrt{m h}} \exp\left\{ -{(m h)^3\over3} -  2a (m h) \right\} \\
&= {\kappa^{1/3}\over \sqrt\pi}\Big(1 + O\Big({A_k\over \kappa^{1/3}}\Big)\Big) \bar I_k \,,
\end{split}
\end{equation*}
where $h= \kappa^{-2/3}$ and $\bar I_k$ is an integral sum of $~\displaystyle{I_k= \int_{m_1 h}^{m_2 h} e^{-x^3/3 - 2a x}\,{dx\over\sqrt x}}\,$.

Since the distance between an integral sum of a monotone function $f(x)$ on $[a, b]$ 
with uniform mesh $x_i= a + i h$  to the integral is less than $h |f(b)-f(a)|$, we have
$\displaystyle{|I_k - \bar I_k| < {h\over \sqrt{m_1 h}} \sim {1\over \sqrt{\kappa}}\,}$.

On the other hand, 
\begin{equation*}
\begin{split}
\left|I_k - \int_0^\i e^{-x^3/3 - 2a x}\,{dx\over\sqrt x}\right| &
< \int_0^{m_1 h}{dx\over\sqrt x} + \int_{m_2 h}^\i e^{-x^3/3}\,dx \vspace*{5ex}\\
& = O\big(\sqrt{ m_1 h} + e^{ -(m_2 h)^3/3}\big) = O\big(k^{-1/6}\big) ,
\end{split}
\end{equation*}
which implies that $~\displaystyle{\bar I_k = \Big(1 +  O\big(k^{-1/6}\big)\Big) \int_0^\i e^{-x^3/3 - 2a x}\,{dx\over\sqrt x}}~$, and hence
$$ 
S_{\kappa}^{(2)}= (A^*)^{-2} \kappa^{1/3} \Big(1 +  O\big(k^{-1/6}\big)\Big) \,. 
$$

Adding to this the estimates of $S_{\kappa}^{(1)}$ and $S_{\kappa}^{(3)}$, we conclude that the same magnitude has the whole sum $S_{\kappa}^*$, which finishes the proof of the theorem in view of \eqref{e5.4} and \eqref{e5.5}. \qed
\section{Concluding remarks}
(1) To obtain lower bounds for $\tau_{n,k}$ from Theorem~\ref{t1.1}, sharper than the one in Corollary~\ref{c1.3}, one needs upper estimates for $S_{n,k+1}(\o)$ and $1-\o^2$. Regarding the first quantity, we point out that from considerations in \cite[\S4]{NNS2018} it follows that
$$
S_{n,k}(x)\leq \frac{(2k)!!}{(2k-1)!!}\approx\frac{\sqrt{\pi}}{2}\,\sqrt{k}\,,\quad x^2\in [0,1-k^2/n^2].
$$

By a result of Driver and Jordaan \cite{DJ2012} (see \cite{Nik2019} for some improvements),  the largest zero of $P_n^{(\l)}$,  $x_{n,n}(\l)$, satisfies  
$$
1-x_{n,n}(\l)^2\leq \frac{(2\l+1)(2\l+3)}{n(n+2\l)+2(\l+1)(2\l+1)}\,.
$$
This yields the following counterpart to the estimate  $1-\o^2\geq \Big(\frac{k+2}{n}\Big)^{2}$: 
\begin{equation*}
1-\omega^2\leq \frac{(2k+3)(2k+5)}{n^2+3k^2+12k+11}
\leq \Big(\frac{k+2}{n}\Big)^{2}\,\frac{4}{1+3\Big(\dfrac{k+2}{n}\Big)^{2}}\,. 
\end{equation*}
We observe that if $k$ is small relative to $n$, then the ratio of the upper and the lower bounds for $1-\omega^2$ is nearly $4$.
\medskip

\noindent
(2) From the proof of Theorem~\ref{t1.4} it is clear that the second exponential term in the approximation of $\tau_k^*$, which is significant, but is missing in the estimate (1.11) in \cite{NNS2018}, comes from the upper estimate for $\o$ chosen there. Actually, the method from \cite{NNS2018} can cover this term and potentially it can overestimate $\tau_{n,k}$ only by a factor $c\,k^{5/12}$ (in the area $n>O(k^{3/2})$). Therefore, the upper estimate (1.5), obtained by combining the main result of \cite{NNS2018} and the exact formula (1.3), overestimates $\tau_{n,k}$ by a factor $c\,k^{5/24}$ in the worst case ($n>>k$).

On the contrary, in the area $k\approx n$ both the estimates in \cite{NNS2018} and those obtained here are sharp with respect to the order of $k$.
\medskip

\noindent
(3) Using the results in \cite[\S5]{TD1999}, we get the asymptotic formula for the largest zero of $P_m^{(p,p)}(x)$, which holds uniformly according to the parameters in the domain 
$0< p < C\,m$ :
$$ x_1= \sqrt{1 - \t p^2} - {a\,\t p^2\over \sqrt{1 - \t p^2}^{1/3}}\,p^{-2/3} + O\big(p^{-4/3}\big) , $$
where $\displaystyle{\t p= {p\over m+p+1/2}}$ and $a$ is the same constant as in Theorem~\ref{t1.4}. Then, applying this to $\o$, in the same manner as in the proof of Theorem~\ref{t1.4} one can obtain the formula
\begin{equation}\label{e6.1}
\tau_{n,k}= A \rho_\l^{n/2} e^{-a(1-\l^2)^{1/3}k^{1/3}}\Big({1\over k^2} - {1\over n^2}\Big)^{1\over 12} \Big(1 + O\big(k^{-1/6}\big) \Big) ,
\end{equation}
where  $\displaystyle{\l=\t p={k+1/2\over n}}$ and $\displaystyle{\rho_\l= \Big({2\over 1 + \l}\Big)^{1 + \l}\Big({1 - \l\over 2}\Big)^{1 - \l} < 1}$ (cf. \cite{NNS2018}). The constant for "$O$"-term in \eqref{e6.1} does not depend on $n$ and $k$, provided $\l < 1 - \d$, i.e. when $k$ is not close to $n$.
\medskip

(4) As mentioned in the introduction, our interest in $\,\tau_{n,k}\,$ is motivated by the role it plays in certain  inequalities of Markov- and Landau-type. However, Petras' result yields an explicit representation for the local maxima of $\big[P_n^{(\l)}\big]^2$, $\,\l\in\mathbb{N}$, and therefore is applicable to the estimation of the largest critical values of the ultraspherical polynomials $P_n^{(\l)}$, 
$\,\l\in \mathbb{N}$.


\begin{thebibliography}{99}
\bibitem{TD1999}
T.\,M. Dunster, Uniform asymptotic approximations for the Jacobi and ultraspherical polynomials, and related functions, {\it Meth. Applic. Analysis} {\bf 6} (1999), 281--316.

\bibitem{DJ2012}
K. Driver, K. Jordaan, Bounds for extreme zeros of some classical orthogonal polynomials.
\emph{J. Approx. Theory} \textbf{164}, 2012, 1200--1204.

\bibitem{AE1953}
A. Erd\'{e}lyi, Higher Transcedental Functions, Vol. 1, A. Erd\'{e}lyi, ed., McGraw-Hill, 1953.

\bibitem{ES1942}
P. Erd\"os and G. Szeg\H{o}, On a problem of I. Schur,
\emph{Ann. Math.} \textbf{43}\,(1942), no. 2, 451--470.

\bibitem{BOE1998}
B.-O. Eriksson, Some best constants in the Landau inequality
on a finite interval, \emph{J. Approx. Theory} \textbf{94}\,(1998),
no. 3, 420--454.

\bibitem{NNS2018}
N. Naidenov, G. Nikolov, A. Shadrin, On the largest critical value
of $T_n^{(k)}$, \emph{SIAM J. Math. Anal.} \textbf{50}(3), 2018,
2389--2408.

\bibitem{Nik2005}
G. Nikolov, Inequalities of Duffin--Schaeffer type II. \emph{East J.
Approx.} \textbf{11}, 2005, 147--168.

\bibitem{Nik2019}
G. Nikolov, New bounds for the extreme zeros of Jacobi polynomials.
\emph{Proc. Amer. Math. Soc.} \textbf{147}(4), 2019, 1541--1550.

\bibitem{KP1996}
K. Petras, An asymptotic expansion for the weights of Gaussian
quadrature formulae. \emph{Acta Math. Hungar.} \textbf{70}(1--2),
1996, 89--100.

\bibitem{QW1999}
C.\,K. Qu and R. Wong, "Best possible" upper and lower bounds for 
the zeros of the Bessel function $J_\nu(x)$, {\it Trans. Amer. Math. Soc.} 
{\bf 351} (1999), 2833--2859.
  
\bibitem{AS2004}
A. Shadrin, Twelve proofs of the Markov inequality. In:
\emph{Approximation Theory: A volume dedicated to Borislav Bojanov}
(D. K. Dimitrov, G. Nikolov, and R. Uluchev, Eds.), Professor Marin
Drinov Academic Publishing House, Sofia, 2004, pp. 233--298.

\bibitem{AS2014}
A. Shadrin, The Landau-Kolmogorov inequality revisited,
\emph{Discrete Contin. Dyn. Syst.} \textbf{34}\,(2014), no. 3,
1183--1210.

\bibitem{SD1938}
A.\,C. Schaeffer and R.\,J. Duffin, On some inequalities of S.
Bernstein and W. Markoff for derivatives of polynomials, \emph{Bull.
Amer. Math. Soc.} \textbf{44}\,(1938), no. 4, 289--297.

\bibitem{GS1975} G. Szeg\H{o}, Orthogonal Polynomials, 4th ed.,
Amer. Math. Soc. Coll. Publ., Vol. 23, Providence, RI, 1975.


\end{thebibliography}
\end{document}